\documentclass[12pt]{article}
\pdfoutput = 1
\usepackage[UKenglish]{babel}
\usepackage[T1]{fontenc}
\usepackage{lmodern,amsmath,amsthm,amsfonts,amssymb,microtype,thmtools,underscore,mathtools,thm-restate}
\usepackage[shortlabels]{enumitem}
\setlist[itemize]{topsep=0ex,itemsep=0ex,parsep=0ex}
\setlist[enumerate]{topsep=0ex,itemsep=0ex,parsep=0ex}
\usepackage[usenames,dvipsnames,svgnames,table]{xcolor}
\usepackage[hyphens]{url} 
\usepackage[linktoc = all, hidelinks, colorlinks, unicode=true, pagebackref=true]{hyperref} 
\hypersetup{linkcolor={blue!70!black}, citecolor={green!70!black}, urlcolor={blue!70!black}, pdftitle = {Quasi-isometry}}
\usepackage[capitalise, compress, nameinlink, noabbrev]{cleveref} 
\crefname{lem}{Lemma}{Lemmas}
\crefname{thm}{Theorem}{Theorems}
\crefname{ques}{Question}{Theorems}
\crefname{cor}{Corollary}{Corollaries}
\crefname{enumi}{Item}{Items}
\crefformat{equation}{(#2#1#3)}
\Crefformat{equation}{Equation #2(#1)#3}
\crefformat{enumi}{#2#1#3}
\Crefformat{enumi}{Item (#2#1#3)}

\newcommand{\defn}[1]{\textcolor{Maroon}{\emph{#1}}}
\usepackage[margin=28mm]{geometry}
\usepackage{parskip}
\renewcommand{\baselinestretch}{1.1}
\allowdisplaybreaks
\usepackage{pgf,tikz,ifthen,calc}
\usepackage{float, graphicx}
\usepackage{tkz-euclide}
\tkzSetUpPoint[fill=black, size = 3pt]
\tkzSetUpLine[color=black, line width=0.6pt]
\tkzSetUpCircle[color=black, line width=0.6pt]

\DeclarePairedDelimiter{\floor}{\lfloor}{\rfloor}
\DeclarePairedDelimiter{\ceil}{\lceil}{\rceil}
\DeclarePairedDelimiter{\abs}{\lvert}{\rvert}
\DeclarePairedDelimiter{\set}{\{}{\}}

\renewcommand{\epsilon}{\varepsilon}
\renewcommand{\emptyset}{\varnothing}
\renewcommand{\geq}{\geqslant}
\renewcommand{\ge}{\geqslant}
\renewcommand{\leq}{\leqslant}
\renewcommand{\le}{\leqslant}

\DeclareMathOperator{\dist}{dist}
\DeclareMathOperator{\tw}{tw}

\newcommand{\NN}{\mathbb{N}}

\newcommand*{\sse}{\subseteq}

\renewcommand{\thefootnote}{\fnsymbol{footnote}}
\theoremstyle{plain}
\newtheorem{thm}{Theorem}
\newtheorem{conj}[thm]{Conjecture}
\newtheorem*{conj*}{Conjecture}
\newtheorem{lem}[thm]{Lemma}

\newtheorem{prop}[thm]{Proposition}

\crefname{obs}{Observation}{Observations}
\newtheorem*{lem*}{Lemma}
\theoremstyle{definition}

\date{}

\begin{document}

\title{\bf\fontsize{18pt}{18pt}\selectfont Fat minors cannot be thinned\\ (by quasi-isometries)}

\author{James Davies\,\footnotemark[1] \quad
	Robert Hickingbotham\,\footnotemark[2]  \\
    Freddie Illingworth\,\footnotemark[3] \quad
    Rose McCarty\,\footnotemark[4] 
	}

\footnotetext[1]{Department of Pure Mathematics and Mathematical Statistics, University of Cambridge, Cambridge, UK (\textsf{\href{mailto:jgd37@cam.ac.uk}{jgd37@cam.ac.uk}}).}

\footnotetext[2]{School of Mathematics, Monash University, Melbourne, Australia (\textsf{\href{mailto:robert.hickingbotham@monash.edu}{robert.hickingbotham@monash.edu}}). Research supported by an Australia Mathematical Society Lift-Off Fellowship and the Monash Postgraduate Publication Award.}

\footnotetext[3]{Department of Mathematics, University College London, London, UK (\textsf{\href{mailto:f.illingworth@ucl.ac.uk}{f.illingworth@ucl.ac.uk}}). Research supported by EPSRC grant EP/V521917/1 and the Heilbronn Institute for Mathematical Research.}

\footnotetext[4]{School of Mathematics and School of Computer Science, Georgia Institute of Technology, Atlanta, USA (\textsf{\href{mailto:rmccarty3@gatech.edu}{rmccarty3@gatech.edu}}). Supported by the National Science Foundation under Grant No.\ DMS-2202961.}

\maketitle
\begin{abstract}
    We disprove the conjecture of Georgakopoulos and Papasoglu that a length space (or graph) with no $K$-fat $H$ minor is quasi-isometric to a graph with no $H$ minor. Our counterexample is furthermore not quasi-isometric to a graph with no 2-fat $H$ minor or a length space with no $H$ minor.
    On the other hand, we show that the following weakening holds: any graph with no $K$-fat $H$ minor is quasi-isometric to a graph with no $3$-fat $H$ minor.
\end{abstract}

\renewcommand{\thefootnote}{\arabic{footnote}}

\section{Introduction}
Inspired by connections to metric geometry, fat minors were independently introduced by Chepoi, Dragan, Newman, Rabinovich, and Vaxes~\cite{chepoi2012constant} and Bonamy, Bousquet, Esperet, Groenland, Liu, Pirot, and Scott~\cite{bonamy2023asymptotic} as a coarse variant of graph minors. 
They have been a key tool in resolving open problems at the intersection of structural graph theory and coarse geometry \cite{bonamy2023asymptotic}.
Georgakopoulos and Papasoglu~\cite{georgakopoulos2023graph} recently gave a systematic overview of this blossoming area which can be described as ``coarse graph theory''. 
At the heart of their paper, they proposed a natural conjecture that would effectively lift graph minor theory to the coarse setting: for every finite graph $H$, graphs with no fat $H$ minor are quasi-isometric to graphs with no $H$ minor. Quasi-isometry is a fundamental notion from geometric group theory which preserves the large-scale geometry of a metric space. So, if true, this conjecture would imply coarse analogues to various results from graph minor theory, such as the Graph Minor Structure Theorem of Robertson and Seymour~\cite{RobertsonSeymourStructThm}. 

Georgakopoulos and Papasoglu's conjecture has been settled for particular graphs $H$. Manning~\cite{manning2005geometry} characterised quasi-trees which implies the case $H = K_3$ (see~\cite{georgakopoulos2023graph}). The case $H = K_{2, 3}$ was proved by Chepoi, Dragan, Newman, Rabinovich, and Vaxes~\cite{chepoi2012constant}, and more generally, Fujiwara and Papasoglu~\cite{fujiwara2023coarse} recently characterised quasi-cacti. The case $H = K_{1, m}$ was recently settled by Georgakopoulos and Papasoglu~\cite{georgakopoulos2023graph}.

Our main contribution is a counterexample to the conjecture of Georgakopoulos and Papasoglu.

\begin{thm}\label{CounterExampleThm}
    There exists a finite graph $H$ such that, for every $q\in \NN$, there is a graph $G_q$ that does not contain $H$ as a $3$-fat minor and is not $q$-quasi-isometric to any graph with no \textup{(}$2$-fat\textup{)} $H$ minor.
\end{thm}

In fact, \cref{CounterExampleThm} gives a strong counterexample to Georgakopoulos and Papasoglu's conjecture, since we can find $H$ as a 2-fat minor, rather than just as a minor.
One might hope that an analogue of the conjecture might still hold for edge weighted graphs or, more generally, length spaces, but this is not the case.

\begin{thm}\label{CounterExampleThmLength}
    There exists a finite graph $H$ such that, for every $q\in \NN$, there is a graph $G_q$ that does not contain $H$ as a $3$-fat minor and is not $q$-quasi-isometric to any length space with no $(2^{-13q^2})$-fat $H$ minor.
\end{thm}

It turns out that \cref{CounterExampleThm} is essentially tight;
we prove that graphs with no $K$-fat $H$ minor are $K$-quasi-isometric to graphs with no $3$-fat $H$ minor.
By scaling, 
\cref{CounterExampleThmLength} is also similarly tight up to some positive factor depending on $q$ (see \cref{scale}).

\begin{thm}\label{weakfavortie2} 
    Let $K$ be a positive integer, $\mathcal{H}$ a non-empty collection of graphs, and $G$ a graph that has no $K$-fat $H$ minor for every $H \in \mathcal{H}$. Then $G$ is $K$-quasi-isometric to a graph $\widehat{G}$ that does not contain a $3$-fat $H$ minor for each $H \in \mathcal{H}$.
\end{thm}

 Our proof of \cref{CounterExampleThm} is based on a recent construction of Nguyen, Scott, and Seymour~\cite{nguyen2024counterexample} which was used to disprove a conjectured coarse version of Menger's theorem. The connection between Menger's theorem and \cref{CounterExampleThm} intuitively comes from the fact that Menger's theorem can be interpreted as a statement about rooted minors.

\section{Preliminaries}

A \defn{length space} is a metric space $(M, d)$ such that, for every $x, y \in M$ and $\epsilon > 0$, there is an $(x, y)$-arc of length at most $d(x, y) + \epsilon$ whenever $d(x, y)$ is finite.
In particular, every graph (where the edges are arcs of length one) is a length space (in which one can even take $\epsilon = 0$). Graphs with arbitrary non-negative real-valued edge-lengths also induce a length space. See \cite{BHMetricSpaces} for further background on length spaces.

For point sets $X, Y$ in a length space $G$, an \defn{$(X, Y)$-path} is a path in $G$ with one end-point in $X$ and the other in $Y$. The \defn{distance} between $X$ and $Y$, denoted \defn{$\dist_G(X, Y)$} is the infimum of the lengths of the $(X, Y)$-paths in $G$. Note that this corresponds to the usual notion of distance for graphs.

We say that a set of points $Y$ in a length space $G$ is \defn{path-connected} if, for every $x,y\in Y$, there is an $(x, y)$-path contained in $Y$.
Note that this corresponds to the usual notion of connectivity for graphs.
For a length space $G$, a non-negative real $r$, and a set of points $Y$, we let $N_G^r[Y]$ denote the set of points $a$ such that there is an $(a, Y)$-path in $G$ of length at most $r$. In graphs this is just the set of points at distance at most $r$ from $Y$.
Note that for any point $y \in G$, $N_G^r[\set{y}]$ is always path-connected.

For a positive integer $K$, a \defn{$K$-fat minor model} of a graph $H$ in a length space $G$ is a collection $(B_v \colon v \in V(H)) \cup (P_e \colon e\in E(H))$ of path-connected sets in $G$ such that
\begin{itemize}
    \item $V(B_v)\cap V(P_e)\neq \emptyset$ whenever $v$ is an end of $e$ in $H$; and
    \item for any pair of distinct $X, Y \in \set{B_v\colon v\in V(H)}\cup \set{P_e \colon e\in E(H)}$ not covered by the above condition, we have $\dist_G(X,Y) \geq K$.
\end{itemize}
If $G$ contains an $K$-fat minor model of $H$, then we say that $H$ is a \defn{$K$-fat minor} of $G$.

For $q \in \NN$, a \defn{$q$-quasi-isometry} of a length space $G$ into a length space $\widehat{G}$ is a map $\phi \colon X \to \widehat{X}$ such that, for every $x,y\in X$,
\begin{equation*}
    q^{-1} \cdot \dist_X(x,y) - q \leq \dist_{\widehat{X}}(\phi(x), \phi(y))\leq q \cdot \dist_X(x,y) + q,
\end{equation*}
and, for every $\widehat{v} \in \widehat{X}$, there is a $v \in X$ with $  \dist_{\widehat{X}}(\widehat{v},\phi(v)) \leq q$. If $\widehat{X}$ is a graph, then we restrict to quasi-isometries $\phi$ with $\phi \colon X \to V(\widehat{X})$.

Let $G$ be a graph and $Z\subseteq V(G)$. We write $G[Z]$ for the subgraph of $G$ induced on $Z$. We refer to $Z$ and $G[Z]$ interchangeably, whenever there is no chance of confusion. A graph $J$ is a \defn{subdivision} of a $G$ if $J$ can be obtained from $G$ by replacing each edge $uv$ of $G$ by a path with end-vertices $u$ and $v$. If each path has exactly $k$ internal vertices, then we say that $J$ is the \defn{$k$-subdivision} of $G$.

For a graph $G$, a \defn{tree-decomposition} of $G$ is a collection of \defn{bags} $(B_x \sse V(G) \colon x \in V(T))$ indexed by a tree $T$ such that for each $v \in V(G)$, $T[\set{x \in V(T) \colon v \in B_x}]$ is a non-empty subtree of $T$; and for each $uv \in E(G)$, there is a node $x \in V(T)$ such that $u,v \in B_x$. The \defn{width} of $(B_x \sse V(G) \colon x \in V(T))$ is $\max\set{\abs{B_x} - 1 \colon x \in V(T)}$. The \defn{treewidth $\tw(G)$} of a graph $G$ is the minimum width of a tree-decomposition of $G$. Treewidth is a fundamental parameter in algorithmic and structural graph theory and is the standard measure of how similar a graph is to a tree.

\section{Removing coarseness counterexample}

\subsection{The construction}

We start with Nguyen, Scott, and Seymour's counterexample to a conjectured coarse version of Menger's theorem~\cite{nguyen2024counterexample}. For a positive integer $q$, they defined the graph $N_q$ shown in \cref{fig:NSS} where the complete binary tree 
has depth $13q^2$ with each solid line being an edge and each dotted line represents a path of length $14q^2$.

\begin{figure}[H]
    \centering
    \includegraphics[width=\textwidth]{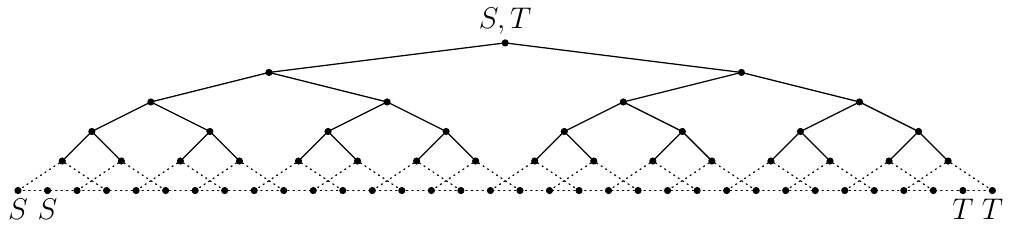}
    \vspace{-24pt}
    \caption{The graph $N_q$; the dotted lines denote paths of length $14q^2$.}
    \label{fig:NSS}
\end{figure}

They proved that $N_q$ has the following property.

\begin{lem}[\cite{nguyen2024counterexample}]\label{NSSproperties}
    Let $S$ and $T$ be the sets of the three vertices in $N_q$ as shown in the figure,
    and let $P_1$, $P_2$, $P_3$ be three $(S,T)$-paths in $N_q$. Then there exist distinct $i, j \in \set{1, 2, 3}$ such that $\dist(P_i,P_j) \leq 2$.
\end{lem}

We prove an additional property of $N_q$.

\begin{lem}\label{TreewidthNSS}
    $N_q$ has treewidth at most seven.
\end{lem}

\begin{proof}
    Subdividing an edge does not change the treewidth of a graph and so $N_q$ has the same treewidth as the graph $N'_q$ where each dotted segment has been replaced by an edge. Contract the following edges of the path at the bottom of $N'_q$: the edges joining the 1st and 2nd vertices, 3rd and 4th vertices, the 5th and 6th vertices, \ldots. The resulting graph $N''_q$ is a minor of the graph depicted in \cref{fig:Halin}.

    \begin{figure}[H]
        \centering
        \includegraphics[width=\textwidth]{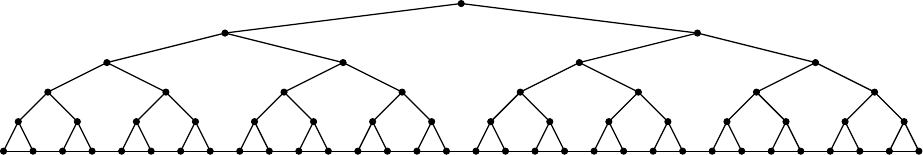}
        \vspace{-24pt}
        \caption{Plane binary tree with path through its leaves.}
        \label{fig:Halin}
    \end{figure}
    
    The graph in \cref{fig:Halin} is a subgraph of a Halin graph\footnote{A \defn{Halin graph} is a planar graph that can be constructed from a planar embedding of a tree by connecting its leaves with a cycle which passes around the tree's boundary.} and so has treewidth at most three~\cite{Bodlaender88}. Thus $\tw(N''_q) \leq 3$. When contracting the edges to get from $N'_q$ to $N''_q$, the bags drop in size by a factor of at most two and so $\tw(N_q) = \tw(N'_q) \leq 7$, as required.
\end{proof}

Now we turn to the construction of $H$ and $G_q$ in \cref{CounterExampleThm}.

Let $L$ and $R$ be vertex-disjoint copies of $K_{15}$ with vertex-sets $\set{x_1, \dotsc, x_{15}}$ and $\set{y_1 \dotsc, y_{15}}$, respectively. Let $H$ be the graph obtained from the union of $L$ and $R$ by adding the edges $x_1 y_1$, $x_2 y_2$, $x_3 y_3$ as shown in \cref{fig:H}. 
We will use the following properties of $H$:
\begin{enumerate}[label=(\roman*)]
    \item $L$ and $R$ are both $4$-connected graphs with treewidth $14$; and
    \item $H$ and $(H \backslash \set{x_2 y_2, x_3 y_3}) \cup \set{x_2 y_3, x_3 y_2}$ are isomorphic.
\end{enumerate}

\begin{figure}[H]
	\centering
	\begin{tikzpicture}
		\tkzDefPoint(-1.8, 0){O1}
		\tkzDefPoint(-0.6, 0){P1}
		\tkzDefPoint(1.8, 0){O2}
		\tkzDefPoint(0.6, 0){P2}
		\tkzDefPoint(-1, -0.5){L3}
		\tkzDefPoint(-1, 0){L2}
		\tkzDefPoint(-1, 0.5){L1}
		\tkzDefPoint(1, -0.5){R3}
		\tkzDefPoint(1, 0){R2}
		\tkzDefPoint(1, 0.5){R1}
		
		\tkzDrawCircle(O1, P1)
		\tkzDrawCircle(O2, P2)
		\tkzDrawSegments(L1,R1 L2,R2 L3,R3)
		\tkzDrawPoints(L1,L2,L3,R1,R2,R3)
		\tkzLabelPoint[left=-4pt](O1){\large$L$}
		\tkzLabelPoint[right=-4pt](O2){\large$R$}
        \tkzLabelPoint[left](L1){$x_1$}
        \tkzLabelPoint[left](L2){$x_2$}
        \tkzLabelPoint[left](L3){$x_3$}
        \tkzLabelPoint[right](R1){$y_1$}
        \tkzLabelPoint[right](R2){$y_2$}
        \tkzLabelPoint[right](R3){$y_3$}
	\end{tikzpicture}
    \caption{The graph $H$; both $L$ and $R$ are $4$-connected with treewidth 14.}
	\label{fig:H}
\end{figure}
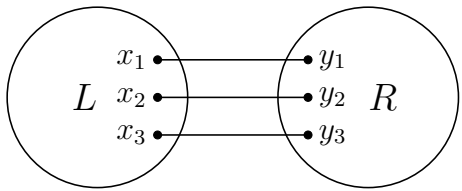

The graph $G_q$ is obtained as follows and shown in \cref{fig:G}.
Let $L^{\star}$ be a $16q^2$-subdivision of $K_{15}$ and let $R^\star$ be a $16q^2$-subdivision of $K_{15}$. We let $x_1^\star, \dotsc, x_{15}^\star$ and $y_1^\star, \dotsc, y_{15}^\star$ denote the high degree vertices of $L^\star$ and $R^{\star}$, respectively. 
Take disjoint copies of each of $L^\star$, $R^\star$, and $N_q$. 
Label the $S$ vertices of $N_q$ as $S_1$, $S_2$, $S_3$ and the $T$ vertices of $N_q$ as $T_1$, $T_2$, $T_3$.
For each $i \in \set{1, 2, 3}$ add a path of length $16q^2$ from $x_i^\star$ to $S_i$ and add a path of length $16q^2$ from $y_i^\star$ to $T_i$.

\begin{figure}[H]
	\centering
	\begin{tikzpicture}
		\foreach \pt in {0,1,2}{
			\tkzDefPoint(\pt*120 + 90:1.5){A\pt}
			\tkzDrawPoint(A\pt)
		}

		\tkzDrawPolygon(A0,A1,A2)

		\tkzDefPoint(-3.6, 0){O1}
		\tkzDefPoint(-2.4, 0){P1}
		\tkzDefPoint(3.6, 0){O2}
		\tkzDefPoint(2.4, 0){P2}
		\tkzDefPoint(-2.8, -0.5){L3}
		\tkzDefPoint(-2.8, 0){L2}
		\tkzDefPoint(-2.8, 0.5){L1}
		\tkzDefPoint(2.8, -0.5){R3}
		\tkzDefPoint(2.8, 0){R2}
		\tkzDefPoint(2.8, 0.5){R1}

		\tkzDefMidPoint(L1,A0) \tkzGetPoint{s1}
		\tkzDefMidPoint(R1,A0) \tkzGetPoint{t1}
		\tkzDefMidPoint(L2,A1) \tkzGetPoint{s2}
		\tkzDefMidPoint(R2,A2) \tkzGetPoint{t2}

		\tkzDefShiftPoint[s2](-80:1){s3}
		\tkzDefShiftPoint[t2](-100:1){t3}

		\tkzDefBarycentricPoint(A1=4,A2=1) \tkzGetPoint{B1}
		\tkzDefBarycentricPoint(A1=1,A2=4) \tkzGetPoint{B2}

		\tkzDrawSegments[dotted](L1,A0 R1,A0 L2,A1 R2,A2 L3,s3 s3,B1 R3,t3 t3,B2)

		\tkzDrawCircle(O1, P1)
		\tkzDrawCircle(O2, P2)
		\tkzDrawPoints(L1,L2,L3,R1,R2,R3)
		\tkzDrawPoints(B1,B2)

        \tkzLabelPoint[above](A0){$S_1, T_1$}
        \tkzLabelPoint[shift={(-0.2,0.75)}](A1){$S_2$}
        \tkzLabelPoint[below](B1){$S_3$}
        \tkzLabelPoint[shift={(0.2,0.75)}](A2){$T_2$}
        \tkzLabelPoint[below](B2){$T_3$}

        \tkzLabelPoint[left](L1){$x_1^\star$}
        \tkzLabelPoint[left](L2){$x_2^\star$}
        \tkzLabelPoint[left](L3){$x_3^\star$}
        \tkzLabelPoint[right](R1){$y_1^\star$}
        \tkzLabelPoint[right](R2){$y_2^\star$}
        \tkzLabelPoint[right](R3){$y_3^\star$}

		\tkzLabelPoint[left=-12pt](0,0){\large$N_q$}
		\tkzLabelPoint[left=-10pt](O1){\large$L^\star$}
		\tkzLabelPoint[right=-10pt](O2){\large$R^\star$}
	\end{tikzpicture}
    \caption{The graph $G_q$.}
	\label{fig:G}
\end{figure}
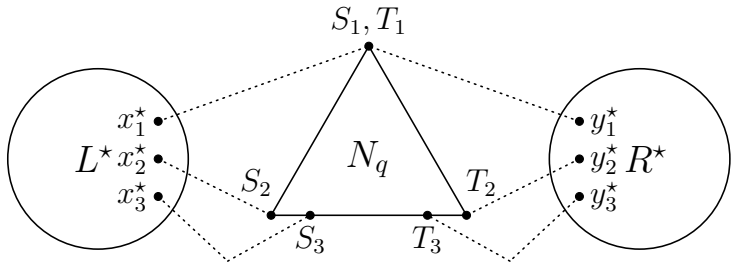

\subsection{Avoiding fat minors}

We now show that $G_q$ does not contain $H$ as a 3-fat minor.

\begin{lem}\label{MainNoFatMinors}
    For every $q\in \NN$, $G_q$ does not contain $H$ as a $3$-fat minor.
\end{lem}

\begin{proof}
    Suppose that $G_q$ contains $H$ as a $3$-fat minor. Let $(B_v\colon v\in V(H))$ and $(P_e\colon e\in E(H))$ be a $3$-fat minor model of $H$ in $G_q$. We begin by considering the $3$-fat sub-model of $L$ in $G_q$.

    \textbf{Claim 1:}
        There exists $z\in V(L)$ such that $B_z\subseteq L^{\star}$ or $B_z\subseteq R^{\star}$.
    \begin{proof}
          Let $\tilde{N}_q=G_q-L^{\star}-R^{\star}$. Since $\tilde{N_q}$ is obtained from $N_q$ by adding paths, each with one end-vertex on $N_q$, it follows that $\tw(\tilde{N}_q)=\tw(N_q)$. Now suppose that $V(B_v)\cap V(\tilde{N}_q)\neq \emptyset$ for every $v\in V(L)$. Let $X = \set{v\in V(L)\colon V(B_v) \not \subseteq V(\tilde{N}_q)}$ and $Y=\set{uv\in E(L-X)\colon V(P_{uv}) \not \subseteq V(\tilde{N}_q)}$. Then $L-X-Y$ is a minor of $\tilde{N}_q$.
          Since every path in $G_q$ from $V(\tilde{N}_q)$ to $V(L^{\star})\cup V(R^{\star})$ contains a vertex from $\set{x_1^\star,x_2^\star,x_3^\star,y_1^\star,y_2^\star,y_3^\star}$, it follows that $|X\cup Y|\leq 6$. Therefore $\tw(\tilde{N}_q)\geq \tw(L-X-Y) \geq \tw(L)-6\geq 8$, contradicting \cref{TreewidthNSS}. As such, there exists $z\in V(L)$ such that $V(B_z)\cap V(\tilde{N}_q)=\emptyset$. Since $B_z$ is connected, it follows that either $B_z\subseteq L^{\star}$ or $B_z\subseteq R^{\star}$.
    \end{proof}
       
    By symmetry, we may assume without loss of generality that $B_z\subseteq L^{\star}$.

    \textbf{Claim 2:}  For every $v\in V(L)$, $V(B_v)\cap V(L^{\star})\neq \emptyset$. 

    \begin{proof}
        Suppose there is some vertex $v\in V(L)$ such that $V(B_v)\cap V(L^{\star})= \emptyset$. Then every $(B_z,B_v)$-path in $G_q$ contains a vertex from $\set{x_1^\star,x_2^\star,x_3^\star}$. So $G_q$ contains at most three vertex-disjoint $(B_z,B_v)$-paths, which contradicts $(B_v\colon v\in V(L))$ and $(P_e\colon e\in E(L))$ being a $3$-fat model of a $4$-connected graph.
    \end{proof}

    \textbf{Claim 3:} Every vertex of degree at least $3$ in $L^{\star}$ is contained in $B_v$ for some $v\in V(L)$.

    \begin{proof}
        Let $v\in V(L)$. Since $\deg_L(v)\geq 3$, $B_v$ contains a vertex in $G_q$ of degree at least $3$. If $V(B_v)\subseteq V(L^{\star})$, then $B_v$ contains a vertex of degree at least $3$ in $L^{\star}$. If $V(B_v)\cap V(N_q)\neq \emptyset$, then Claim 2 implies that $B_v$ contains $x_1^{\star}$, $x_2^{\star}$, or $x_3^{\star}$. So each $B_v$ contain a vertex of degree at least $3$ in $L^{\star}$. Since $\abs{\set{u\in V(L^{\star})\colon \deg_{L^{\star}}(u)\geq 3}} = \abs{V(L)}$, the claim then follows.
    \end{proof}

     We now consider the $3$-fat sub-model of $R$ in $G_q$.
    
    \textbf{Claim 4:} For every $u\in V(R)$, $V(B_u)\cap V(R^{\star})\neq \emptyset$.

    \begin{proof}
        By symmetry, Claims 1 \& 2 imply that either $V(B_u)\cap V(L^{\star})\neq \emptyset$ for every $u \in V(R)$, or $V(B_u)\cap V(R^{\star})\neq \emptyset$ for every $u \in V(R)$. Suppose that $V(B_u)\cap V(L^{\star})\neq \emptyset$ for every $u \in V(R)$. Then there is a vertex $u\in V(R)$ such that $B_u\subseteq L^{\star}$. Since $\deg_R(u)\geq 3$, it follows that $B_u$ contains a vertex of degree at least $3$ from $L^{\star}$. But by Claim 3, this implies that $V(B_u)\cap V(B_v)\neq \emptyset$ for some $v\in V(L)$, a contradiction. Therefore, $V(B_u)\cap V(R^{\star})\neq \emptyset$ for every $u \in V(R)$.
    \end{proof}

    To conclude the proof, consider the edges $x_1y_1,x_2y_2,x_3y_3\in E(H)$.
    For each $i\in \set{1,2,3}$, let $P_i$ be the path in our model that corresponds to $x_iy_i$. Since $B_{x_i}\cup P_i\cup B_{y_i}$ is connected, Claims 2 \& 4 imply that there is an $(L^{\star},R^{\star})$-path $P_i'$ in $B_{x_i}\cup P_i\cup B_{y_i}$. By \cref{NSSproperties}, there are distinct $i,j \in \set{1, 2, 3}$ such that $\dist_{G_q}(P_i',P_j')\leq 2$. Therefore, there is $X_i \in \set{B_{x_i}, P_i, B_{y_i}}$ and $X_j \in \set{B_{x_j}, P_j, B_{y_j}}$ such that $\dist_{G_q}(X_i,X_j)\leq 2$, contradicting $(B_v\colon v\in V(H))$ and $(P_e\colon e\in E(H))$ being $3$-fat.
\end{proof}

\subsection{Finding fat minors}

In this subsection, we shall find $H$ as a fat minor in length spaces (or graphs) that are quasi-isometric to $G_q$. 

\begin{lem}\label{LemmaFindingFat}
    Let $\widehat{G}$ be a length space that is $q$-quasi-isometric to $G_q$.
    If $\widehat{G}$ is a graph, then let $K = 2$, otherwise let $K = 2^{-13q^2}$.
    Then $\widehat{G}$ contains $H$ as a $K$-fat minor.
\end{lem}

\Cref{MainNoFatMinors,LemmaFindingFat} immediately implies \Cref{CounterExampleThm,CounterExampleThmLength}. Before proving \cref{LemmaFindingFat}, we need the following two lemmas to help us build our fat $H$ model in $\widehat{G}$. The first lemma allows us to find path-connected sets in $\widehat{G}$ which spans a set of elements, while the second lemma allows us to find path-connected sets in $\widehat{G}$ that are sufficiently far away from each other.

\begin{lem}\label{PathsQuasiIsometry}
    Let $q\in \NN$ and $G$ be a graph. Let $\widehat{G}$ be a length space that is $q$-quasi-isometric to $G$ with quasi-isometry $\phi$.
    For every connected induced subgraph $G[X]$ of $G$, the set $\widehat{X}= N_{\widehat{G}}^{q + 1}[\phi(X)]$ in $\widehat{G}$ is path-connected.
\end{lem}

\begin{proof}
    Since $G[X]$ is connected, it is enough to show that if $uv\in E(G)$, then $N_{\widehat{G}}^{q + 1}[\phi(u)] \cup N_{\widehat{G}}^{q + 1}[\phi(v)]$ is path-connected.
    Clearly both $N_{\widehat{G}}^{q + 1}[\phi(u)]$ and $N_{\widehat{G}}^{q + 1}[\phi(v)]$ are path-connected.
    Observe that $\dist_{\widehat{G}}(\phi(u), \phi(v)) \le q \dist_G(u,v) +q=2q$.
    Since $\widehat{G}$ is a length space, there is a path $P$ in $\widehat{G}$ between $\phi(u)$ and $\phi(v)$ of length at most $2q + 1$.
    Then, for each $p\in P$, $\min\set{\dist_{\widehat{G}}(p, \phi(u)), \dist_{\widehat{G}}(p, \phi(v))} \leq q + \frac{1}{2} < q + 1$, and so $P\subseteq N_{\widehat{G}}^{q + 1}[\phi(u)] \cup N_{\widehat{G}}^{q + 1}[\phi(v)]$. Thus $N_{\widehat{G}}^{q + 1}[\phi(u)] \cup N_{\widehat{G}}^{q + 1}[\phi(v)]$ is path-connected.
\end{proof}

\begin{lem}\label{closesets}
    Let $\widehat{G}$ be a length space containing path-connected sets $X_1, \dotsc, X_n$.
    Then for every $\epsilon > 0$, there exists some integer $1\le m \le n$ and path-connected sets $Y_1, \dotsc, Y_m$ pairwise at distance at least $\epsilon$ such that $\bigcup_{i=1}^n X_i \subseteq \bigcup_{i=1}^m Y_i \subseteq \bigcup_{i=1}^n N^{(n-m)\epsilon}[X_i]$.
\end{lem}

\begin{proof}
    The lemma is trivially true if $n=1$, so we shall proceed inductively on $n$ assuming that the lemma holds for $n-1$.
    If $X_1, \dotsc, X_n$ are pairwise at distance at least $\epsilon$, then we can simply take $m=n$ and $Y_i=X_i$ for $1\le i \le m$.
    So, without loss of generality, we may assume that the distance between $X_{n-1}$ and $X_n$ is strictly less than $\epsilon$.
    Now, let $X_{n-1}'=N^\epsilon[X_{n-1}] \cup X_n$.
    Then $X_{n-1}'$ is a path-connected set with $ X_{n-1} \cup X_n \subseteq X_{n-1}' \subseteq N^{\epsilon}[X_{n-1}] \cup N^{\epsilon}[X_{n}]$.
    Applying the inductive hypothesis to $X_1, \dotsc, X_{n-2} , X_{n-1}'$, it now follows that there exists some integer $1\le m \le n-1$ and path-connected sets $Y_1, \dotsc, Y_m$ pairwise at distance at least $\epsilon$ such that 
    \[
    \bigcup_{i=1}^n X_i \subseteq 
    \bigcup_{i=1}^{n-2} X_i \cup X_{n-1}' \subseteq
    \bigcup_{i=1}^m Y_i \subseteq 
    \bigcup_{i=1}^{n-2} N^{(n-1-m)\epsilon}[X_i]
    \cup N^{(n-1-m)\epsilon}[X_{n-1}']
    \subseteq
    \bigcup_{i=1}^n N^{(n-m)\epsilon}[X_i],
    \]
    as desired.
\end{proof}

We are now ready to prove \cref{LemmaFindingFat}.

\begin{proof}[Proof of \cref{LemmaFindingFat}]
    Let $\phi$ be a $q$-quasi-isometry from $G_q$ to $\widehat{G}$. For each connected induced subgraph $G_q[A]$ of $G_q$, let $\widehat{A} = N_{\widehat{G}}^{q + 1}[\phi(A)]$. By \cref{PathsQuasiIsometry}, each such $\widehat{A}$ is a path-connected set in $\widehat{G}$.
    Observe that if $G_q[A_1]$, $G_q[A_2]$ are connected induced subgraphs of $G_q$ with $\dist_{G_q}(A_1,A_2) \ge 8q^2$, then $\dist_{\widehat{G}}(\widehat{A}_1, \widehat{A}_2)
    \ge q^{-1} \dist_{G_q}(A_1,A_2) - q - 2(q + 1)
    \ge 8q - q - 2(q + 1)
    \ge 2$ (since $q\geq 1$). We shall use this fact repeatedly.
    Throughout the proof, the reader may wish to refer to \cref{fig:findingminors}.

    Recall, in the definition of $H$ and $G_q$, that the 15-cliques $L$ and $R$ have vertex-sets $\set{x_1, \dotsc, x_{15}}$ and $\set{y_1, \dotsc, y_{15}}$ and that their $16q^2$-subdivisions $L^\star$ and $R^\star$ have high degree vertices $x_1^\star, \dotsc, x_{15}^\star$ and $y_1^\star, \dotsc, y_{15}^\star$, respectively.
    We being by construction a $K$-fat minor model of $L \cup R = H \setminus \set{x_1 y_1, x_2 y_2, x_3 y_3}$ in $\widehat{G}$.
    For $1\le i \le 15$, let $L_i = N_{G_q}^{4q^2}[x_i^\star]$ and $R_i = N_{G_q}^{4q^2}[y_i^\star]$. For each pair $1 \le i < j \le 15$, let $L_{i, j}$ be the shortest path in $L^\star \subset G_q$ between $L_i$ and $L_j$ and $R_{i, j}$ be the shortest path in $R^\star \subset G_q$ between $R_i$ and $R_j$.

    For distinct pairs $(i,j)$ and $(i',j')$, we have that $\dist_{G_q}(L_{i,j}, L_{i',j'}) \ge 8q^2$.
    Therefore, $\dist_{\widehat{G}}(\widehat{L}_{i, j}, \widehat{L}_{i',j'}) \ge 2$.
    Similarly, for each pair $i \neq j$, we have that $\dist_{G_q}(L_i, L_j) = 8q^2$, and therefore $\dist_{\widehat{G}}(\widehat{L}_i, \widehat{L}_j) \ge 2$.
    Furthermore, for each triple $i,j,k$ with $i<j$, $k \not\in \set{i,j}$ we have that $\dist_{G_q}(L_{i,j},L_{k}) \ge 8q^2$, and therefore $\dist_{\widehat{G}}(\widehat{L}_{i, j}, \widehat{L}_k) \ge 2$. This all holds similarly for the sets $R_i$, $R_{i,j}$. Furthermore, for $X\in \{L_i,L_{i,j}\colon 1\leq i<j\leq 15\}$ and $Y\in \{R_i,R_{i,j}\colon 1\leq i<j\leq 15\}$, we clearly have $\dist_{\widehat{G}}(\widehat{X},\widehat{Y})\geq 2$. So $(\widehat{L}_i,\widehat{R}_i\colon 1\leq i\leq 15)$ and $(\widehat{L}_{i,j},\widehat{R}_{i,j}\colon 1\leq i<j\leq 15)$ is a 2-fat minor model of $H \setminus \set{x_1 y_1, x_2 y_2, x_3 y_3}$ in $\widehat{G}$.

    We now define a path in $\widehat{G}$ which will correspond to $x_1 y_1 \in E(H)$ in our model.
    Recall the vertices $S_1$, $S_2$, $S_3$, $T_1$, $T_2$, $T_3$ that appear in the construction of $G_q$ (see \cref{fig:G}).
    For each $1 \le a \le 3$, let $L_a'$ be the shortest path in $G_q$ between $L_a$ and $S_a$.
    Similarly to before, for each $1\le a \le 3$, we have that $\dist_{\widehat{G}}(\widehat{L}'_a, \widehat{L}_i)\ge 2$ for all $1\le i \le 15$ with $i\neq  a$, and $\dist_{\widehat{G}}(\widehat{L}'_a, \widehat{L}_{i, j})\ge 2$ for all $1\le i<j\le 15$. We define $R'_1$, $R'_2$, $R'_3$ similarly. For each $1\le a \le 3$, we have that $\dist_{\widehat{G}}(\widehat{R}'_a, \widehat{R}_i)\ge 2$ for all $1\le i \le 15$ with $i \neq a$, and $\dist_{\widehat{G}}(\widehat{R}_a', \widehat{R}_{i, j})\ge 2$ for all $1\le i<j\le 15$.
    Let $Z_1 = \widehat{L}'_1 \cup \widehat{R}'_1$.
    Adding $Z_1$ to our fat minor model gives a 2-fat minor model of $H \setminus \set{x_2 y_2, x_3 y_3}$.

    \begin{figure}[H]
        \centering
        \includegraphics[width=\textwidth]{"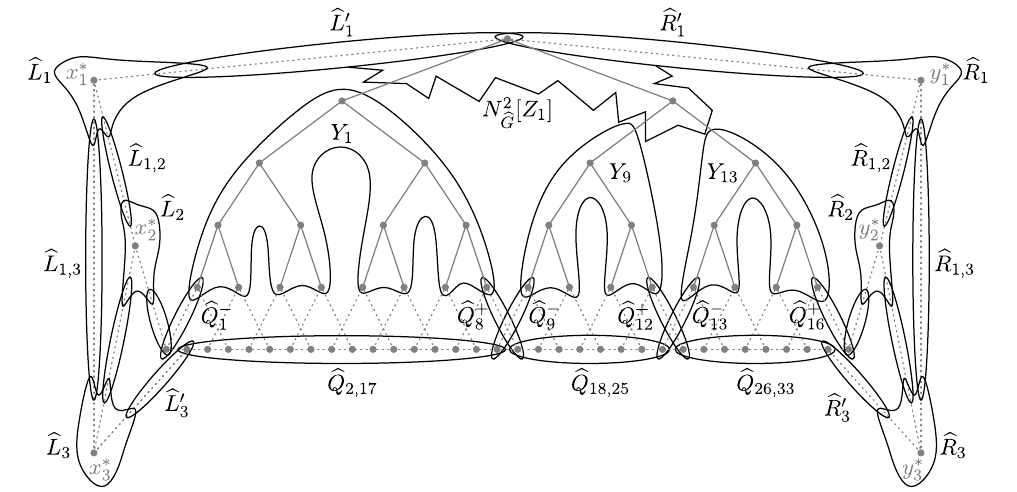"}
        \caption{The length space $\widehat{G}$ (black) on top of the corresponding objects in $G_q$ (grey). We have $Z_1 = \widehat{L}'_1 \cup \widehat{R}'_1$,
        $Z_2 = \widehat{L}_2' \cup \widehat{Q}^-_{1} \cup Y_1\cup \widehat{Q}_8^+ \cup \widehat{Q}_{18,25} \cup \widehat{Q}_{13}^- \cup Y_{13} \cup \widehat{Q}_{16}^+ \cup \widehat{R}_2'$, and 
        $Z_3 = \widehat{L}_3' \cup \widehat{Q}_{2,17} \cup \widehat{Q}_9^- \cup Y_9 \cup \widehat{Q}_{12}^+ \cup \widehat{Q}_{26,33} \cup \widehat{R}_3'$.
        }
        \label{fig:findingminors}
    \end{figure}

    Recall the construction of $N_q$ shown in \cref{fig:NSS}. Let $B_q$ denote the complete binary tree of height $13q^2$ (this is the graph consisting of solid edges) and let $\ell_1, \dotsc, \ell_{2^{13q^2}}$ be its leaves from left to right. Let $p_1, \dotsc, p_{2^{13 q^2 + 1} + 2}$ denote the thick vertices on the bottom path from left to right (so $p_1, p_2 \in S$ and $p_{2^{13 q^2 + 1} + 1}, p_{2^{13 q^2 + 1} + 2} \in T$).
    Let $B'_q = N_{N_q}^{7q^2}[V(B_q)]$; that is, $B'_q$ is the graph obtained from a complete binary tree of height $13q^2 + 1$ by subdividing every edge that is incident to a leaf $7q^2-1$ times. For each $1 \le i \le 2^{13q^2}$, let $\ell_i^-$ be the leaf of $B'_q$ descended from $\ell_i$ to the left, and let $\ell_i^+$ be the leaf of $B'_q$ descended from $\ell_i$ to the right.
    Let $B_q^\ast = \widehat{B'_q} \setminus N_{\widehat{G}}^3[Z_1]$. 
    For each $1\le i \le 2^{13q^2}$ and each vertex $w$ on the unique path of $B'_q$ between $\ell_i^-$ and $\ell_i^+$, we have
    $\dist_{\widehat{G}}(\phi(w), Z_1) \ge q^{-1}\dist_{G_q}(w, L_1'\cup R_1')-(q+1) -q
    \ge
    13q-2q-1
    >
    (q+1) +3$.
    Therefore, for each such $w$, we have that $N_{\widehat{G}}^{q+1}[\phi(w)]\subseteq B_q^*$.
    In particular, for each $1 \le i \le 2^{13q^2}$, we have $\phi(\ell_i)\in B_q^*$. For each $1 \le i \le 2^{13q^2}$, let $X_i$ be the component of $B_q^*$ that contains $\phi(\ell_i)$. Then for each $1 \le i \le 2^{13q^2}$, by \cref{PathsQuasiIsometry}, we have $\phi(\ell_i^-),\phi(\ell_i^+)\in X_i$.

    If $\widehat{G}$ is a graph, then set $Y_i=X_i$ for each $1\le i \le 2^{13q^2}$. In this case, if $Y_i\neq  Y_j$, then $\dist_{\widehat{G}}(Y_i,Y_j) \ge 2 = K$.
    This is the only place where we need to treat the case that $\widehat{G}$ is a graph separately. When $\widehat{G}$ is not a graph, we must define $Y_1, \dotsc, Y_{2^{13q^2}}$ differently, as follows.
    By \cref{closesets}, there exist path-connected sets $Y'_1, \dotsc, Y'_m$ pairwise at distance at least $2^{-13q^2}$ in $\widehat{G}$ such that $\bigcup_{i=1}^{13q^2} X_i \subseteq \bigcup_{i=1}^m Y'_i \subseteq \bigcup_{i = 1}^{13q^2} N^{1}_{\widehat{G}}[X_i]
    \subseteq N^1_{\widehat{G}}[N^{q + 1}_{\widehat{G}}[\phi(B'_q)]\setminus N_{\widehat{G}}^3[Z_1]]
    = N^{q+2}_{\widehat{G}}[\phi(B'_q)]\setminus N_{\widehat{G}}^2[Z_1]$.
    Now, for each $1\le i \le 2^{13q^2}$, let $Y_i=Y_j'$, where $\phi(\ell_i)\in Y_j'$.
    Then for every pair $Y_i,Y_j$ with $Y_i \neq Y_j$, we have that $\dist_{\widehat{G}}(Y_i,Y_j) \ge 2^{-13q^2} = K$.
    Note that for each $1\le i \le 2^{13q^2}$, we have $\dist_{\widehat{G}}(Y_i, Z_1) \ge 2 \ge K$.
    We no longer need to distinguish between the cases of whether or not $\widehat{G}$ is a graph.

    Recall that $p_1, \dotsc, p_{2^{13q^2 + 1} + 2}$ are the thick vertices on the bottom path in \cref{fig:NSS}. 
    For each $1 \le i < j \le 2^{13q^2 + 1} + 2$, let $Q_{i, j}$ be the sub-path of the bottom path from $p_i$ to $p_j$
    Observe that, for each tuple $1\le i < j < i' < j' \le 2^{13q^2+1} + 2$, we have $\dist_{\widehat{G}}(\widehat{Q}_{i,j}, \widehat{Q}_{i',j'}) \ge 2$.
    For each $1 \le i < j \le 2^{13q^2 + 1} +2$ and $1 \le k \le 2^{13q^2}$, we also have $\dist_{\widehat{G}}(\widehat{Q}_{i,j}, Y_k) \ge \dist_{\widehat{G}}(\widehat{Q}_{i,j}, N^{q + 2}_{\widehat{G}}[\phi(B'_q)])
    = \dist_{\widehat{G}}(\widehat{Q}_{i,j}, \phi(B'_q)) - q - 2 \ge q^{-1}\dist_{G_q}(Q_{i,j},B'_q) - 2q - 2 \ge 5q - 2 \ge 2$.
    For each $1\le i \le 2^{13q^2}$, let $Q_i^-$ be the path between $\ell_i^-$ and $p_{2i-1}$ in $N_q \setminus B_q$ and let $Q_i^+$ be the path between $\ell_i^+$ and $p_{2i + 2}$ in $N_q \setminus B_q$.
    Similarly, for $1\le i < j \le 2^{13q^2}$, $W_i \in \set{Q_i^-, Q_i^+}$ and $W_j \in \set{Q_j^-, Q_j^+}$, we have $\dist_{\widehat{G}}(\widehat{W}_i, \widehat{W}_j) \ge 2$.
    Furthermore, for $1 \le i < j \le 2^{13q^2 + 1} +2$ and $1 \le k \le 2^{13q^2}$, if either $2k - 1 < i$ or $2k - 1 > j$, then $\dist_{\widehat{G}}(\widehat{Q}_{i, j}, \widehat{Q}_{k}^-) \ge 2$, and similarly, if either $2k + 2 < i$ or $2k + 2 > j$, then $\dist_{\widehat{G}}(\widehat{Q}_{i,j}, \widehat{Q}_{k}^+) \ge 2$.
    
    For $1\le i ,k \le 2^{13q^2}$, if $Y_k \neq Y_i$, then we have $\dist_{\widehat{G}}(\widehat{Q}_i^-, Y_k)
    \ge 
    \dist_{\widehat{G}}(\widehat{Q}_i^-, B_q^* \backslash Y_i)
    \ge 
    \dist_{\widehat{G}}(\widehat{Q}_i^-, \widehat{B}_q' \backslash Y_i)
    \ge 
    \dist_{\widehat{G}}(\widehat{Q}_i^-, \widehat{B_q'\backslash J_i})$, where $J_i$ is the set of vertices of $B_q'$ on the path between $\ell_i$ and $\ell_i^-$ and the last inequality holds since $\widehat{J}_i\subseteq Y_i$. 
    Expanding on this, we further have that 
    \begin{equation*}
        \dist_{\widehat{G}}(\widehat{Q}_i^-, \widehat{B_q'\backslash J_i}) \ge q^{-1} \dist_{G_q}(Q_i^-, B_q'\backslash J_i) -q-2(q+1) \ge 7q -(3q+2) \ge 2.
    \end{equation*}
    Thus, for $i$ and $k$ with $Y_i \neq Y_k$, $\dist_{\widehat{G}}(\widehat{Q}_i^-, Y_k) \geq 2$ and, by symmetry, $\dist_{\widehat{G}}(\widehat{Q}_i^+, Y_k) \geq 2$.

    For each $1\le i \le 2^{13q^2}$, let $e(i)$ be maximum such that $\phi(\ell_{e(i)})\in Y_i$.
    Choose $a_1, \dotsc, a_m$ such that $a_1=1$, $e(a_m) = 2^{13q^2}$, and $a_{i+1} = e(a_i)+1$ for each $1\le i < m$.
    For each $1\le i \le m$, let $D_i = \widehat{Q}_{a_i}^- \cup Y_{a_i} \cup \widehat{Q}_{e(a_i)}^+$.
    Observe that for each pair $1\le i < j \le m$, we have that $\dist_{\widehat{G}}(D_i,D_j)
    \ge K$.
    For each $1\le i \le m$, let $F_i = \widehat{Q}_{2a_i, 2a_i + 1}$. Observe that $\dist(F_i,F_j)\ge 2$ for each pair $1\le i<j\le m$ and $\dist(F_i,D_j)\ge 2$ for each pair $1\le i,j\le m$ with $\abs{i - j} \neq 1$.
    For each $1\le i \le m$, both $F_i$ and $D_i$ are path-connected. Let $D_{m+1}=F_{m+1}= \emptyset$.
    Notice that $Z = \bigl(\bigcup_{i=1}^m D_i \bigr) \cup \bigl(\bigcup_{i=1}^m F_i\bigr) $ has two path-connected components $Z_2'$, $Z_3'$ with $Z_2' = \bigcup_{i=1}^{\ceil{m/2}} (D_{2i-1} \cup F_{2i})$ and $Z_3' = \bigcup_{i=1}^{\ceil{m/2}} (F_{2i - 1} \cup D_{2i})$, 
    and $\dist(Z_2', Z_3') \ge K$.
    If $m$ is odd, then $\phi(S_2),\phi(T_2) \in Z_2'$ and $\phi(S_3),\phi(T_3) \in Z_3'$, and let $Z_2 = \widehat{L}_2' \cup  Z_2' \cup \widehat{R}_2'$ and $Z_3 = \widehat{L}_3' \cup  Z_3' \cup \widehat{R}_3'$.
    If $m$ is even, then $\phi(S_2),\phi(T_3) \in Z_2'$ and $\phi(S_3), \phi(T_2) \in Z_3'$, and let $Z_2 = \widehat{L}_2' \cup  Z_2' \cup \widehat{R}_3'$ and $Z_3= \widehat{L}_3' \cup Z_3' \cup \widehat{R}_2'$.
    Finally, by adding $Z_2$ and $Z_3$ to our fat minor model, we obtains a $K$-fat minor model of $H$ if $m$ is odd, and of $(H \backslash \set{x_2 y_2, x_3 y_3}) \cup \set{x_2 y_3, x_3 y_2}$ if $m$ is even.
    Since $H$ and $(H \backslash \set{x_2 y_2, x_3 y_3}) \cup \set{x_2 y_3, x_3 y_2}$ are isomorphic, $\widehat{G}$ contains $H$ as a $K$-fat minor, as desired.
\end{proof}

\section{Reducing coarseness}

In this section we prove \cref{weakfavortie2}.
For a graph $G$ and a positive integer $K$, let \defn{$G^K$} denote the power graph obtained from $G$ by adding an edge between every pair of distinct vertices at distance at most $K$. The importance of power graphs to coarse graph theory was previously observed by Seymour and the fourth author; see~\cite[Theorem~4]{albrechtsen2023induced}.

\begin{proof}[Proof of \cref{weakfavortie2}.]
    Let $G$ be a graph that does not contain a $K$-fat $H$ minor for every $H \in \mathcal{H}$. Let $\widehat{G} = G^K$. Then the identity map from $V(G)\to V(\widehat{G})$ is a $K$-quasi-isometry since, for each pair of vertices $x,y\in V(G)$, we have 
    $K^{-1} \cdot \dist_G(x,y) \leq \dist_{\widehat{G}}(x,y)\leq \dist_G(x,y)$. It remains to show that $\widehat{G}$ contains no 3-fat $H$ minor for each $H \in \mathcal{H}$.

    Suppose that $\widehat{G}$ contains $H$ as a $3$-fat minor for some $H\in \mathcal{H}$. Let $(B_v\colon v\in V(H))$ and $(P_e \colon e\in E(H))$ be a $3$-fat minor model of $H$ in $\widehat{G}$. We work towards constructing a $K$-fat minor model of $H$ in $G$. Note that, for distinct $u, v \in V(H)$, $\dist_{G^K}(B_u, B_v) \geq 3$ and so $\dist_G(B_u, B_v) \geq 2K + 1$.
    For each $v\in V(H)$, let $B'_v=\set{x\in V({G})\colon \dist_{G}(x, B_v)\leq K/2}$. Then $G[B_v']$ is connected. Furthermore, for distinct $u, v\in V(H)$, it follows that 
    \begin{equation*}
        \dist_G(B'_u, B'_v) \geq \dist_G(B_u, B_v) - 2 \floor{K/2} \geq K + 1.
    \end{equation*}
    For each $e \in E(H)$, let $X_e = \set{x \in V({G}) \colon \dist_{G}(x, P_e) \leq K/2}$. Then, as above, $G[X_e]$ is connected and, for distinct $e, f \in E(H)$, $\dist_G(X_e, X_f) \geq \dist_G(P_e, P_f) - 2 \floor{K/2} \geq K + 1$. Similarly, for each $e \in E(H)$ and $v \in V(H)$ that are not incident, $\dist_G(X_e, B'_v) \geq \dist_G(P_e, B_v) - 2 \floor{K/2} \geq K + 1$.

    Let $e = uv \in E(H)$. By definition of fat minor model, $B_u$ and $P_e$ intersect and so $B'_u$ and $X_e$ intersect (and by symmetry so do $B'_v$ and $X_e$). Let $P'_e$ be a path in $G[X_e]$ between $B'_u \cap X_e$ and $B'_v \cap X_e$. Then $(B'_v \colon v\in V(H))$ and $(P'_e \colon e\in E(H))$ is a $K$-fat minor model of $H$ in $G$, giving the desired contradiction.
\end{proof}

Using a simple scaling argument, we now show an analogous tightness result for \cref{CounterExampleThmLength}.

\begin{prop}\label{scale}
    For any collection $\mathcal{H}$ of graphs, there exists a function $f \colon \mathbb{N} \to \mathbb{N}$ such that, if $(M,d)$ is a length space with no $K$-fat $H$ minor for any $H \in \mathcal{H}$, then, for every $0 < K' < K$, $(M, d)$ is $(K/K')$-quasi-isometric to a length space that contains no $K'$-fat $H$ minor for any $H \in \mathcal{H}$.
\end{prop}

\begin{proof}
    Clearly $(M, d)$ is $(K/K')$-quasi-isometric to $(M, d/(K/K'))$.
    Furthermore, for $H\in \mathcal{H}$, if $\set{B_v\colon v\in V(H)} \cup \set{P_e \colon e\in E(H)}$ is a $K'$-fat minor model of $H$ in $(M, d/(K/K'))$, then $\set{B_v\colon v\in V(H)} \cup \set{P_e \colon e\in E(H)}$ is a $K$-fat minor model of $H$ in $(M,d)$.
    Since $(M,d)$ contains no $K$-fat minor, it follows that $(M, d/(K/K'))$ contains no $K'$-fat minor.
\end{proof}

\section{Concluding remarks and conjectures}

We conclude with some remarks on our construction of the graphs $H$ and $G_q$ in \cref{CounterExampleThm,CounterExampleThmLength}. First, observe that the graph $G_q$ has maximum degree $15$. Now graphs with bounded maximum degree are quasi-isometric to graphs with maximum degree $3$ (this can be achieved by replacing each vertex by a bounded height binary tree). Since excluding a graph as a fat minor is an invariant under quasi-isometry (with slight change in parameter), it follows that \cref{CounterExampleThm} implies that Georgakopoulos and Papasoglu's conjecture is false even if we assume that the graph with no $3$-fat $H$ minor has maximum degree $3$.

Second, in our construction of the graph $H$, instead of choosing $L$ and $R$ to be copies of $K_{15}$, we could instead take $L$ and $R$ to be copies of any $4$-connected graph with treewidth at least $14$ that contains a automorphism that fixes $x_1$ (respectively $y_1$) and swaps $x_2$ with $x_3$ (respectively $y_2$ with $y_3$). In particular, one can choose $L$ and $R$ to be copies of a planar graph, and thus enabling one to make $H$ planar. So the conjecture of Georgakopoulos and Papasoglu is false even when $H$ is planar.

Let us remark that one particularly interesting case of Georgakopoulos and Papasoglu's~\cite{georgakopoulos2023graph} conjecture that remains open is for graphs that exclude both $K_5$ and $K_{3,3}$ as fat minors; it is conjectured that any such graph is quasi-isometric to a planar graph~\cite{georgakopoulos2023graph}, which would imply a coarse analogue of Kuratowski's theorem~\cite{kuratowski1930probleme}.
As $q$ increases, $N_q$ contains increasingly fat $K_5$ and $K_{3,3}$ minors, and so we currently see no way to adapt our construction to disprove the coarse Kuratowski conjecture.

Now, despite Georgakopoulos and Papasoglu's~\cite{georgakopoulos2023graph} conjecture being false, there are natural weaker statements which would allow some results from graph minor theory to be lifted to the coarse setting. We mention two such results. Firstly, one may weaken the conjecture so that the graph forbidden as a fat minor is different to the graph forbidden as a minor.

\begin{conj}\label{weakfavorite}
    For any finite graph $H$, there exists a finite graph $H'$ and a function $f_H \colon \NN \to \NN$ \textup{(}which depends only on $H$\textup{)} such that if $G$ is a graph or length space with no $K$-fat $H$ minor, then $G$ is $f_H(K)$-quasi-isometric to a graph with no $H'$ minor.
\end{conj}

Secondly, one might hope that Georgakopoulos and Papasoglu's conjecture is true in the special case where a 2-fat minor is forbidden.

\begin{conj}\label{weakfavorite3}
    Let $G$ be a graph, and let $\mathcal{H} = \set{H_1,\dots, H_k}$ be a collection of finite graphs. Then $G$ has no $2$-fat $H$ minor for every $H \in \mathcal{H}$ if and only if $G$ is $q$-quasi-isometric to a graph that does not contain an $H$ minor for each $H \in \mathcal{H}$, where $q$ depends on $\mathcal{H}$ only.
\end{conj}

A graph $G$ contains another graph $H$ as an \defn{induced minor} if $H$ can be obtained from $G$ by deleting vertices, contracting edges and then removing all loops and multiple edges. Note that a graph $G$ contains $H$ as a 2-fat minor if and only if $G$ contains the $3$-subdivision of $H$ as an induced minor (assuming $H$ has minimum degree at least $2$).
Thus one could phrase~\cref{weakfavorite3} in terms of excluded induced minors. A potential counterexample to the conjectured induced Menger's theorem\footnote{The induced Menger's conjecture states that for every graph $G$, for every $S, T \subseteq V(G)$ and positive integer, there exists a positive integer $r$ such that either there exist $k$ pairwise anticomplete $S$--$T$ paths in $G$, or there exists a set $Z \subseteq V(G)$ of size at most $k - 1$ such that $N^r[Z]$ intersects every $S$--$T$ path.} may be useful for disproving~\cref{weakfavorite3} by using a similar approach to our disproof of Georgakopoulos and Papasoglu's~\cite{georgakopoulos2023graph} conjecture.
We remark that some weakenings of the induced Menger's conjecture do hold~\cite{gartland2023induced,hendrey2023induced}.

\subsubsection*{Acknowledgements}
This work was completed at the 11th Annual Workshop on Geometry and Graphs held at Bellairs Research Institute in March 2024.
We are grateful to the organisers and participants for providing a stimulating research environment.

{
\fontsize{11pt}{12pt}
\selectfont
	
\hypersetup{linkcolor={red!70!black}}
\setlength{\parskip}{2pt plus 0.3ex minus 0.3ex}

\bibliographystyle{Illingworth.bst}
\newcommand{\etalchar}[1]{$^{#1}$}

\end{document}